\pdfoutput=1
\documentclass[a4paper, 12pt]{amsart}

\usepackage{graphicx}
\usepackage{amsmath}
\usepackage{amsfonts}
\usepackage{amsthm}
\usepackage{amssymb}
\usepackage[table]{xcolor}
\usepackage{stmaryrd}
\usepackage{bm}
\usepackage{tikz-cd}
\usepackage{enumitem}
\usepackage{soul}
\usepackage{mathtools}

\usepackage[utf8]{inputenc} 
\usepackage[T1]{fontenc}             
\usepackage{lmodern}
\usepackage[numbers]{natbib}
\usepackage{framed}
\usepackage{stackrel}
\usepackage{array}
\usepackage{mathabx}

\usepackage{xifthen}

\usepackage[margin=4cm]{geometry}
\usepackage{booktabs}
\usepackage{longtable}

\newcommand{\calA}{\mathcal{A}}

\newcommand{\calM}{\mathcal{M}}

\newcommand{\calO}{\mathcal{O}}

 \newcommand{\R}{\mathbb{R}}
 \newcommand{\C}{\mathbb{C}}
 \newcommand{\Q}{\mathbb{Q}}
 \newcommand{\Z}{\mathbb{Z}}

 \newcommand{\inj}{\hookrightarrow}

 \newcommand{\surj}{\twoheadrightarrow}

\newcommand{\sep}{\: | \:}
\let\oldforall\forall
\renewcommand{\forall}{\; \oldforall}
\let\oldexists\exists
\renewcommand{\exists}{\; \oldexists}

\newcommand{\la}{\langle}
\newcommand{\ra}{\rangle}

\newcommand{\bs}{\backslash}

\DeclareMathOperator{\Stab}{Stab}
\DeclareMathOperator{\PO}{PO}
\DeclareMathOperator{\Ad}{Ad}
\DeclareMathOperator{\tr}{tr}

\theoremstyle{plain} \newtheorem*{theorem*}{Theorem}
\theoremstyle{plain} \newtheorem*{conjecture*}{Conjecture}
\theoremstyle{plain} \newtheorem*{lemma*}{Lemma}
\theoremstyle{plain} \newtheorem*{corollary*}{Corollary}
\theoremstyle{plain} \newtheorem*{proposition*}{Proposition}

\theoremstyle{plain} \newtheorem{thm}{Theorem}[section]
\theoremstyle{plain} \newtheorem{theorem}[thm]{Theorem}
\theoremstyle{plain} \newtheorem{lemma}[thm]{Lemma}
\theoremstyle{plain} \newtheorem{corollary}[thm]{Corollary}
\theoremstyle{plain} \newtheorem{proposition}[thm]{Proposition}
\theoremstyle{definition} \newtheorem{definition}[thm]{Definition}
\theoremstyle{definition} 

\theoremstyle{plain} 
\theoremstyle{plain} 
\theoremstyle{plain} 
\theoremstyle{plain} 
\theoremstyle{plain} 
\theoremstyle{definition} 
\theoremstyle{definition} 
\theoremstyle{plain} 

\theoremstyle{plain} \newtheorem{thm2}{Yolo}[section]
\theoremstyle{plain} 
\theoremstyle{plain} \newtheorem*{claim*}{Claim}
\theoremstyle{definition} \newtheorem{remark}[thm]{Remark}
\theoremstyle{definition} \newtheorem{example}[thm2]{Example}

\theoremstyle{plain}

\theoremstyle{plain} \newtheorem*{theoreme*}{Theorème}
\usepackage{booktabs}
\usepackage{longtable}

\theoremstyle{plain} 
\theoremstyle{plain} 
\theoremstyle{plain} \newtheorem*{lemme*}{Lemme}
\theoremstyle{plain} 
\theoremstyle{plain} \newtheorem*{corollaire*}{corollaire}
\theoremstyle{definition} 

\theoremstyle{plain} 
\theoremstyle{definition} 
\theoremstyle{definition}

\newcommand{\Hy}{\mathbf{H}}
\newcommand{\tM}{\widetilde M}
\newcommand{\G}{\mathbf{G}}
\newcommand{\aPO}{\mathbf{PO}}
\newcommand{\aO}{\mathbf{O}}
\newcommand{\aGL}{\mathbf{GL}}
\newcommand{\bd}{\partial}
\newcommand{\<}{\langle}
\renewcommand{\>}{\rangle}
\renewcommand{\:}{\colon}
\renewcommand{\sep}{\mid}
\newcommand{\Mgl}{\calM^{\mathrm{gl}}}
\newcommand{\Agl}{\calA^{\mathrm{gl}}}
\DeclareMathOperator{\id}{id}

\title{The trace field of hyperbolic gluings}

\author{Olivier Mila}
\address{Centre de recherches mathématiques,
	Université de Montréal,\linebreak
  Pavillon André-Aisenstadt,
	Montréal, Québec, H3T~1J4, Canada}
  \email{olivier.mila@umontreal.ca, \textit{\fontfamily{\familydefault}\selectfont Web: }crm.umontreal.ca/\textasciitilde mila}

\begin{document}
\begin{abstract}
We determine the adjoint trace field of gluings of general hyperbolic manifolds.
This provides a new method to prove the nonarithmeticity of gluings, which can be applied 
to the classical construction of Gromov and Piatetski-Shapiro (and generalizations) as well as certain 
gluings of pieces of commensurable arithmetic manifolds.
As an application we give many new examples of nonarithmetic gluings and prove 
that the unique nonarithmetic Coxeter 5-simplex is not commensurable to 
any gluing of arithmetic pieces.
\end{abstract}
\thanks{This work is supported by the Swiss National Science Foundation, Project numbers \texttt{PP00P2\_157583} and 
\texttt{P2BEP2\_188144}}
\vspace*{-.2cm}
\maketitle
\section{Introduction} \label{sec:intro}

Introduced in 1987 by Gromov and Piatetski-Shapiro \cite{GPS}, the gluing construction 
remains as of today the only 
known method to produce nonarithmetic hyperbolic manifolds in arbitrary dimensions.
More recently, Raimbault \cite{Raimbault} and Gelander and Levit \cite{GL} generalized 
it to show the existence of ``many'' non-commensurable manifolds of bounded volume.
Also, the construction of Belolipetsky and Thomson \cite{BT} (generalizing ideas of Agol \cite{Agol}) 
of manifolds with arbitrarily short systole uses doubling --- a particular type of gluing.

In the study of general hyperbolic manifolds such as those, a central role is played by 
\emph{algebraic} invariants --- the most famous being
the \emph{invariant trace field} and \emph{invariant quaternion algebra} in dimension 2 and 3 
(see \cite{MR}).
In his paper \cite{Vinberg}, Vinberg introduces similar invariants for Zariski-dense 
subgroups of linear groups: the \emph{adjoint trace field} and the \emph{ambient group}
(see Section~\ref{sec:background} for definitions).
By Borel's Density Theorem, any lattice in $\PO(n,1)^\circ$ is Zariski-dense, and thus the adjoint trace field 
can be seen as a generalization of the invariant trace field to hyperbolic manifolds of higher dimensions.
For an arithmetic manifold, the adjoint trace field and the ambient group coincide with the field and 
algebraic group used in its definition (see \cite[Prop.~2.5]{PR}), and for Coxeter groups, Vinberg himself 
computes these invariants in Section~4 of his paper \cite{Vinberg}. 

This article is concerned with the computation of the adjoint trace field (hereafter trace field) and the ambient group for manifolds
arising from gluings.
After defining gluings in their general form (Definition~\ref{def:gluings}) we show that the 
\emph{gluing isometry} (the isometry used to identify the boundary hypersurfaces) naturally admits a \emph{field of definition} 
(Lemma~\ref{lem:field-of-def}).
We then prove (Theorem~\ref{th:tf-gluing-field-of-def-gluing-isom}) that this field of definition coincides with the trace field of the gluing.

In Section~\ref{sec:classical-gluings} we use Theorem~\ref{th:tf-gluing-field-of-def-gluing-isom} to compute the trace field of 
the ``classical'' gluings of 
Gromov--Piatetski-Shapiro and their generalizations \cite{GPS,Raimbault,GL} (assuming \emph{canonical} gluing isometries, see 
Section~\ref{sec:classical-gluings}).
Note that these gluings all use pieces of non-commensurable arithmetic manifolds (recall that a \emph{piece} of a hyperbolic manifold 
$M$ is the completion of $M -  N_1  \cup \cdots \cup N_r$, where the $N_i$ are disjoint finite-volume hypersurfaces in $M$).
Also, we prove that doubling does not increase the trace field (Proposition~\ref{prop:double-same-trace-field}) and thus the manifolds 
of Agol--Belolipetsky--Thomson \cite{Agol,BT} all have the same trace field as the original arithmetic manifolds (this was already 
proven in \cite{Thomson}).

An important application of our work is that it provides a new method to prove nonarithmeticity of certain gluings.
It leads to particularly interesting and surprising results when applied to gluings of pieces of \emph{commensurable} manifolds, 
as the next theorem shows:
\begin{theorem} \label{th:intro-commensurable-large-trace-field}
  Let $k = \Q$ or $\Q(\sqrt{2})$.
  For each even $n=2m \geq 4$, there exists an $n$-manifold $M$ with trace field 
  $K$ of arbitrary large degree $d = [K:k]$ obtained as 
  a gluing of pieces of \emph{pairwise commensurable} arithmetic manifolds with trace field $k$.
  In particular, if $d > 1$, $M$ is nonarithmetic.
\end{theorem}

In fact one can also obtain nonarithmetic gluings using only one piece.
If $M$ is a hyperbolic manifold containing a non-separating hypersurface $N \subset M$, we call a \emph{twist} of $M$ any manifold 
obtained by cutting $M$ along $N$ and gluing it back together using a different isometry $N \to N$ (see Section~\ref{sec:examples}).
We show:
\begin{theorem} \label{th:intro-twist-nonarithmetic}
  Let $K = k(\sqrt{a})$ be a totally real quadratic extension of the number field $k$.
  Assume either that $k \neq \Q$ and $a \notin \Q$, or that $k=\Q$ and $a \in \Z$ is odd.
  Then for each even $n = 2m \geq 4$ there exists an $n$-manifold $M$ with trace field $K$ obtained by twisting an 
  arithmetic manifold with trace field $k$.
  In particular, $M$ is nonarithmetic.
\end{theorem}
It is clear from both theorems that the \emph{gluing isometry} plays an essential role for nonarithmeticity; 
this contrasts with the methods used in \cite{GPS,Raimbault,GL}
(see Section~\ref{sec:classical-gluings}).
As these gluings use pieces of commensurable manifolds and have increased trace field, they are new examples of nonarithmetic 
manifolds, both in the compact and non-compact setting.
Observe also that the twists from Theorem~\ref{th:intro-twist-nonarithmetic} have the same volume as the original arithmetic 
manifolds.

More precise versions of these results are given in Theorem~\ref{th:realization-trace-fields}, 
Examples~\ref{ex:gluings-Q-0-mod-4} and \ref{ex:gluings-Qsqrt2} for Theorem~\ref{th:intro-commensurable-large-trace-field} and Examples~\ref{ex:1} and \ref{ex:2} for Theorem~\ref{th:intro-twist-nonarithmetic}.
Observe that the restrictions imposed on the fields $k$ and $K$ are only technical assumptions used to simplify the proofs, and 
could probably be removed with more careful computations.
However, the fact that these manifolds have \emph{even} dimension is necessary: in Theorem~\ref{th:arith-gluing-odd-dim} we prove 
that gluing \emph{odd dimensional} pieces of commensurable arithmetic manifolds never increases the trace field.

An application of this even-odd dimensional dichotomy is the following.
Let $\Delta_5$ denote the Coxeter $5$-simplex with the following diagram:
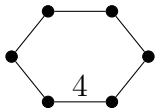
\[
\pgfmathsetmacro{\minsize}{0.15cm}
\begin{tikzpicture}[scale=0.6]
\tikzstyle{every node}=[draw,shape=circle,minimum size=\minsize,inner sep=0, fill=black]
\node at (0,1) (1) {};
\node at (0.8,0) (2) {};
\node at (2.2,0) (3) {};
\node at (3,1) (4) {};
\node at (2.2,2) (5) {};
\node at (0.8,2) (6) {};
\draw (1)--(2) ;
\draw (2)--(3) node[midway,above=1pt, opacity=0, text opacity=1]  {4};
\draw (3)--(4) ;
\draw (4)--(5) ;
\draw (5)--(6) ;
\draw (6)--(1) ;
\end{tikzpicture}
\]
It was shown by Fisher et.\ al.\ \cite{FLMS} that the orbifold $\Delta_5 \bs \Hy^5$ is not commensurable to a gluing of pieces of
\emph{non-commensurable} arithmetic manifolds.
Using Theorem~\ref{th:arith-gluing-odd-dim} we improve this result by showing:
\begin{theorem}\label{th:intro-Delta-5}
  The orbifold $\Delta_5 \bs \Hy^5$ is not commensurable to any gluing of pieces of arithmetic manifolds, commensurable or not.
\end{theorem}
It is not known if such objects exist in higher dimensions.
In particular, Questions~0.4 of \cite{GPS} are still open.

The paper is organized as follows.
In Section~\ref{sec:background} we recall the necessary background, we define gluings, gluing isometries and 
their fields of definition.
Section~\ref{sec:main-theorem} is devoted to the proof of the main theorem (Theorem~\ref{th:tf-gluing-field-of-def-gluing-isom}).
In Section~\ref{sec:classical-gluings}, we compute the trace fields of the classical gluing constructions of \cite{GPS,Raimbault,GL} 
and prove Theorem~\ref{th:intro-commensurable-large-trace-field}.
In Section~\ref{sec:corollaries} we expose various corollaries such as:
\begin{itemize}
  \item Proposition~\ref{prop:gluing-containing-noncommensurable-nonarith}: 
    A gluing containing (among others) two pieces of non-commensurable arithmetic 
    manifolds is nonarithmetic.
  \item Theorem~\ref{th:gluing-quadratic}: Each gluing operation increases the trace field by at most a quadratic extension.
  \item Proposition~\ref{prop:double-same-trace-field}: Doubling does not change the trace field.
\end{itemize}
Finally, in Section~\ref{sec:examples} we prove Theorems~\ref{th:intro-twist-nonarithmetic} and \ref{th:intro-Delta-5}.

\emph{Acknowledgements:}
Most of the results in this article are contained in the author's PhD thesis \cite{Mila-PhD}.
The author thanks Vincent Emery for many helpful suggestions and discussions, and Jean Raimbault and Matthew Stover for 
their comments.

\section{Gluings and fields of definition} \label{sec:background}
In this paper all manifolds are hyperbolic and complete of finite volume.
Furthermore, manifolds with boundary are assumed to have totally geodesic boundary components that are themselves of finite volume.
Let $M$ be a manifold with (possibly empty) boundary.
Then its universal cover $\tM$ is either $\Hy^n$ or an intersection of half spaces in $\Hy^n$.
Thus $M$ can be written as $\Gamma \bs \tM$, where $\Gamma \subset \PO(n,1)$ is a discrete torsion-free subgroup stabilizing $\tM$.
If $M$ has boundary then $\Gamma$ is not a lattice, but it is still \emph{almost Zariski-dense} in $\PO(n,1)$, by which we mean that 
$\Gamma \cap \PO(n,1)^\circ$ is Zariski-dense in $\PO(n,1)^\circ$ (see \cite[Cor.~1.7b]{GPS}).
If $N \subset \bd M$ is a hypersurface, then its pre-image in $\tM$ consists of a disjoint union of 
hyperplanes in $\bd \tM$; we call those \emph{hyperplane lifts} of $N$.


The \emph{(adjoint) trace field} of $M = \Gamma \bs \tM$ is the field 
$\Q(\{\tr\Ad \gamma \sep \gamma \in \Gamma)\})$, where 
$\Ad$ denotes adjoint representation.
It is invariant under commensurability \cite[Th.~3]{Vinberg}.
By \cite[Th.~1]{Vinberg} this is also the smallest field $k$ such that there exists an algebraic $k$-group $\G$ and an isomorphism $\G(\R) \cong \PO(n,1)^\circ$ such that 
$\Gamma \cap \PO(n,1)^\circ \subset \G(k)$ via this isomorphism.
By Zariski-density, $\G$ is unique up to $k$-isomorphism; we call it the \emph{ambient group} of $\Gamma\cap \PO(n,1)^\circ$.
Since $\Gamma$ is contained in a lattice, its trace field is always a number field (see \cite[Chap.~1, §6]{Vinberg-Lie-Groups}).

Assume that $M$ contains (in its interior or in its boundary) an immersed hypersurface $N$.
Then the Zariski-closure of $\Stab_{\Gamma\cap \PO(n,1)^\circ} N$ is a $k$-form of $\PO(n-1,1)^\circ$ in $\G$.
Thus the ambient group of $\Gamma \cap \PO(n,1)^\circ$ must be of the form $\aPO_f^\circ$, where $f$ is a quadratic form over $k$ (see \cite{LiMillson}).
Hence the above isomorphism extends to $\aPO_f(\R) \cong \PO(n,1)$, such that the image of $\Gamma$ lies in $\aPO_f(k)$.
We will (by abuse of terminology) call $\aPO_f$ the ambient group of $\Gamma$.

For a quadratic form $f$ of signature $(n,1)$, let $\Hy_f$ denote the $f$-hyper\-boloid model for hyperbolic space:
$\Hy_f = \{x \in \R^{n+1} \sep f(x) = -1\} / \{\pm 1\}$.
Its group of isometries is $\aPO_f(\R)$.
Then one can see the universal cover $\tM$ in $\Hy_f$ in such a way that $M = \Gamma \bs \tM$, with $\Gamma \subset \aPO_f(k)$.
We will call a \emph{model} for $M$ any choice of universal cover $\tM \subset \Hy_f$ and monodromy representation $\Gamma \subset \aPO_f(k)$ such that $M = \Gamma \bs \tM$.

We are now ready to define gluings. 
\begin{definition} \label{def:gluings}
  \begin{enumerate}
    \item Let $M_1, M_2$ be manifolds with boundary having isometric boundary components $N_1 \subset \bd M_1$ and $N_2 \subset \bd M_2$.
      Let $\varphi\:N_1 \to N_2$ be an isometry.
      Then the adjunction space \linebreak $M_1 \cup_\varphi M_2 = M_1 \sqcup M_2 / \sim_\varphi$, where $\sim_\varphi$ is the relation generated by $x \sim_\varphi \varphi(x)$ 
      for $x \in N_1$
      is called the \emph{interbreeding} of $M_1$ and $M_2$ along $\varphi$.
    \item Let $M_1$ be a manifold with boundary, and let $N_1, N_2 \subset \bd M_1$ be isometric boundary components.
      Let $\varphi\:N_1 \to N_2$ be an isometry.
      Then the quotient space $M_1 / \sim_\varphi$, with $\sim_\varphi$ the equivalence relation generated by $x \sim_\varphi \varphi(x)$ for $x \in N_1$
      is called the \emph{closing up} of $M_1$ along $\varphi$.
    \item If $\calM$ is a set of manifolds with boundary, we let $\Mgl$ denote the smallest set 
      containing $\calM$ and closed under interbreedings and closing ups.
      Elements of $\Mgl$ are called \emph{gluings} of $\calM$.
  \end{enumerate}
\end{definition}
In both (1) and (2), we call $\varphi$ the \emph{gluing isometry}.
It follows from \cite[§2.10.B]{GPS} that gluings have a complete structure of hyperbolic manifold with boundary.

Let $M$ be a manifold. 
A \emph{piece} of $M$ is the manifold with boundary obtained by taking the completion of 
$M - N_1 \cup \cdots \cup N_r$, where the $N_i$ are disjoint hypersurfaces.
If $\calA$ denotes the set of all pieces of arithmetic manifolds (or \emph{arithmetic pieces} for short), then the manifolds constructed 
by Gromov and Piatetski-Shapiro \cite{GPS} as well as their generalizations mentioned in the introduction are all 
elements of $\Agl$.

We now define extensions of gluing isometries.
Let $M_1, M_2$ be manifolds with boundary, and let $N_1 \subset \bd M_1$ and $N_2 \subset \bd M_2$ be 
isometric boundary components.
Here we allow the case $M_1 = M_2$ in order to treat both interbreeding and closing up at once.
For $i=1,2$, let $k_i$ be the trace field of $M_i$.
Choose models $\tM_i \subset \Hy_{f_i}$ for $M_i$ with monodromy representations 
$\Gamma_i \subset \aPO_{f_i}(k_i)$ and hyperplane lifts $R_i \subset \partial \tM_i$ of $N_i \subset M_i$.
Observe that both cases $M_1 = M_2, \tM_1 \neq \tM_2$ and $M_2 = $ other copy of $M_1, \tM_1 = \tM_2$ are allowed.

An \emph{extension} of the gluing isometry $\varphi\: N_1 \to N_2$ is an isometry 
$\phi\:\Hy_{f_1} \to \Hy_{f_2}$ sending $R_1$ to $R_2$ and such that the following diagram commutes (where the vertical maps are the
coverings):
\[
  \begin{tikzcd}
    R_1 \dar[two heads] \rar["\phi"] & R_2 \dar[two heads] \\
    N_1 \rar["\varphi"] & N_2.
  \end{tikzcd}
\]
In terms of monodromy representations, this implies in particular that:
\[
\phi \Stab_{\Gamma_1}(R_1) \phi^{-1} = \Stab_{\Gamma_2}(R_2).
\]

Observe that such an isometry is always induced by some $A_\phi \in \aGL_n(\R)$ such that 
$f_2 \circ A_\phi = f_1$.
Let $k$ denote the composition of the fields $k_1$ and $k_2$.
Then the base extensions $\aPO_{f_1,k}$ and $\aPO_{f_2,k}$ are algebraic groups defined over $k$ and 
conjugation by $A_\phi$ induces an isomorphism of algebraic groups
$\Phi\: (\aPO_{f_1,k})_\R \to (\aPO_{f_2,k})_\R$.
The \emph{field of definition} of the extension $\phi$ is then the field of definition of the 
algebraic morphism $\Phi$, i.e., the smallest field $K$ with $k \subset K \subset \R$ over which 
$\Phi$ is defined (this field exists by \cite[Cor.~4.8.11]{EGA42}).

The next lemma (whose proof will be given in the next section) shows that this field actually only depends 
on the gluing isometry $\varphi$.
\begin{lemma} \label{lem:field-of-def}
  Any two extensions of a gluing isometry have the same field of definition.
\end{lemma}
This justifies the following definition:

\begin{definition}
  The \emph{field of definition} of a gluing isometry is the field of definition of any extension.
\end{definition}

\section{Main Theorem} \label{sec:main-theorem}

The goal of this section is to prove the main theorem of this article:
\begin{theorem}\label{th:tf-gluing-field-of-def-gluing-isom}
  Let $M$ be constructed by interbreeding $M_1$ with $M_2$ or closing up $M_1$.
  Then the trace field of $M$ coincides with the field of definition of the gluing isometry.
\end{theorem}

First we prove Lemma~\ref{lem:field-of-def}.
In what follows, we will often make use of the elementary fact that a morphism $\Phi\:\aPO_{f_1,\R} \to \aPO_{f_2,\R}$ sending an 
almost Zariski-dense subgroup $\Gamma \subset \aPO_{f_1}(k)$ into $\aPO_{f_2}(k)$ is itself defined over $k$ (see \cite[Prop.~2.13]{Mila-PhD}).
\begin{proof}[Proof of Lemma~\ref{lem:field-of-def}]
  Let $M_1, M_2$ be manifolds with boundary (as before, we allow the case $M_1 = M_2$ to include 
  closing ups).
  Let $\varphi\: N_1 \subset \bd M_1 \to N_2 \subset \bd M_2$ be a gluing isometry, let $\phi, \phi'$ be two extensions and let $\Phi, \Phi'$ be the corresponding 
  isomorphisms on the ambient groups.
  By definition, $\phi$ and $\phi'$ can be defined using arbitrary models for $M_1$ and $M_2$.
  The first step is to reduce the proof to the case where both $\phi$ and $\phi'$ are defined 
  on the same fixed models for $M_1$ and $M_2$.

  Let $\tM \subset \Hy_{f}$ and $\tM' \subset \Hy_{f'}$ be two models for $M_1$, with corresponding monodromy 
  representations $\Gamma \subset \aPO_f(k_1)$ and $\Gamma' \subset \aPO_{f'}(k_1)$ where $k_1$ is the trace 
  field 
  of $M_1$.
  By uniqueness of universal covers, there exists an isometry $\psi_1\: \tM \to \tM'$ which extends to an isometry
  $\psi_1\: \Hy_{f} \to \Hy_{f'}$, in such a way that 
  $\Psi_1(\Gamma) = \psi_1 \,\Gamma \,\psi^{-1}_1 = \Gamma'$, 
  where $\Psi_1 \:\aPO_{f,\R} \to \aPO_{f',\R}$ is the isomorphism induced by $\psi_1$.
  Since $\Gamma$ is (almost) Zariski-dense in $\aPO_f$ and $\Gamma' \subset \aPO_f(k)$, the isomorphism $\Psi_1$ is defined over $k_1$.
  Similarly, for $M_2$ any two choices of models give rise to such an isomorphism $\Psi_2$ of the ambient groups, which is defined over $k_2$, the trace field of $M_2$.

  Now changing the models of $M_1$ and $M_2$ in an extension $\phi$ of $\varphi$ amounts to replacing the induced isomorphism $\Phi$ with 
  $\Psi_1 \circ \Phi \circ \Psi_2^{-1}$, for $\Psi_1$, $\Psi_2$ defined as above.
  Since $\Psi_i$ is defined over $k_i$, and since the field of definition $k$ of $\Phi$ contains (by definition) 
  both $k_1$ and $k_2$, we see that $k$ equals the field of definition of $\Psi_1 \circ \Phi \circ \Psi_2^{-1}$.

  Therefore, for $i=1,2$, we can fix a model $\tM_i \subset \Hy_{f_i}$ for $M_i$ with monodromy representation $\Gamma_i \subset \aPO_{f_i}(k_i)$, and assume that 
  both extensions $\phi$ and $\phi'$ are maps $\Hy_{f_1} \to \Hy_{f_2}$.
  Moreover, since all hyperplane lifts $R_i \subset \bd \tM_i$ of $N_i \subset \bd M_i$ are in the same $\Gamma_i$-orbit, we can assume (up to replacing $\phi'$ by 
  $\gamma_2 \circ \phi' \circ\gamma_1$ for some $\gamma_1 \in \Gamma_1$ and $\gamma_2 \in \Gamma_2$) that there are hyperplane lifts $R_1 \subset \bd M_1$ of $N_1$ and 
  $R_2 \subset \bd M_2$ of $N_2$ such that both $\phi$ and $\phi'$ send $R_1$ to $R_2$.
  Observe that since conjugation by $\gamma_i \in \Gamma_i$ is an operation defined over $k_i$, the same argument as before implies that we do not change the 
  field of definition of $\Phi'$.

  Now since both $\phi$ and $\phi'$ are extensions of the same $\varphi\: N_1 \to N_2$,
  the isometry $\phi^{-1} \circ \phi'$ fixes $R_1$, and induces the trivial isometry $N_1 \to N_1$.
  Therefore, $\phi^{-1} \circ \phi' = \gamma$ or $\rho_1 \gamma$, where 
  $\gamma \in \Stab_{\Gamma_1}(R_1)$ and $\rho_1$ 
  is the reflection about $R_1$.
  Since $\gamma$ fixes $R_1$, we can again replace $\phi'$ with $\phi' \circ \gamma^{-1}$ 
  (as before without changing the field of definition of $\Phi'$) 
  to obtain that $\phi^{-1} \circ \phi' \in \{1, \rho_1\}$.
  It follows that $\Phi' = \Phi \circ \Psi$, where $\Psi \in \{\id, \text{conjugation by } \rho_1\}$.
  Since $\rho_1 \in \aPO_{f_1}(k_1)$ by \cite[Lem.~2.1]{Emery-Mila} we see that the fields of definition of $\Phi$ and $\Phi'$ coincide, as desired.
\end{proof}

Before going into the proof of Theorem \ref{th:tf-gluing-field-of-def-gluing-isom}, we give a small lemma.

\begin{lemma} \label{lem:trace-field-inclusion}
  Let $M_1$ be a piece of the manifold $M$.
  Let $k_1$ (resp.\ $k$) and $\aPO_{f_1}$ (resp.\ $\aPO_{f}$) denote the trace field and the ambient group of $M_1$ (resp.\ $M$).
  Then $k_1 \subset k$ and $\aPO_f \cong \aPO_{f_1, k}$.
\end{lemma}
\begin{proof}
  The inclusion of the trace fields is clear from the definition.
  For the ambient groups, let $M_1 = \Gamma_1 \bs \tM_1$ (resp.\ $M = \Gamma \bs \tM$) be models for 
  $M_1$ (resp.\ $M$) with $\Gamma_1 \subset \aPO_{f_1}(k_1)$ (resp.\ $\Gamma \subset \aPO_f(k)$).
  Then the inclusion $M_1 \inj M$ lifts to an isometry $\varphi\: \Hy_{f_1} \to \Hy_f$ 
  sending $\tM_1$ inside $\tM$ and such that $\varphi \Gamma_1 \varphi^{-1} = \Stab_{\Gamma}(\tM_1)$.
  Hence the isomorphism $\Phi\:\aPO_{f_1,\R} \to \aPO_{f,\R}$ induced by $\varphi$ sends 
  $\Gamma_1 \subset \aPO_{f_1}(k_1) \subset \aPO_{f_1}(k)$ into $\Gamma \subset \aPO_f(k)$.
  Since $\Gamma_1$ is (almost) Zariski-dense, $\Phi$ is defined over $k$.
\end{proof}

We are ready for the:
\begin{proof}[Proof of Theorem \ref{th:tf-gluing-field-of-def-gluing-isom}]
  As before, we set $M_2 = M_1$ in the closing up case.
  Let $\varphi$ denote the gluing isometry.
  For $i = 1,2$ let $\tM_i \subset \Hy_{f_i}$ be a model for $M_i$ with associated monodromy representation 
  $\Gamma_i \subset \aPO_{f_i}(k_i)$, where $k_i$ is the trace field of $M_i$, and let $\tM \subset \Hy_f$, 
  $\Gamma \subset \aPO_f(K)$ be defined similarly for $M$.
  Let $N_i \subset M_i$ denote the gluing hypersurfaces, and $R_i \subset \partial \tM_i$ some fixed hyperplane lifts.
  Finally let $N \subset M$ denote the glued hypersurface.
 
  By (the proof of) Lemma~\ref{lem:trace-field-inclusion}, $k_1 \subset K$ and we can assume 
  (up to replacing $f$ with $f_1$) that $\tM_1 \subset \tM \subset \Hy_{f_1}$ and that 
  the corresponding monodromy representation $\Gamma \subset \aPO_{f_1}(K)$ is such that
  $\Gamma_1 = \Stab_\Gamma(\tM_1)$.
  In particular we see that $R_1$ is a hyperplane lift of $N \subset M$.

  Similarly, the same argument gives an isometry $\alpha\: \Hy_{f_2} \to \Hy_{f_1}$ such that 
  $\alpha \Gamma_2 \alpha^{-1} = \Stab_\Gamma(\alpha \tM_2)$.
  We can make further assumptions on $\alpha$: in the interbreeding case we can choose it such that 
  $\alpha R_2 = R_1$ and $\Gamma = \la  \Gamma_1,  \alpha \Gamma_2 \alpha^{-1} \ra$, and in the closing up case,
  since $M_2 = M_1$, we can pick $\alpha$ so that $\alpha \tM_2 = \tM_1$ and 
  $\alpha \Gamma_2 \alpha^{-1} = \Gamma_1$.

  Now we define an isometry $\phi\: \Hy_{f_1} \to \Hy_{f_2}$.
  In the interbreeding case, we just set $\phi = \alpha^{-1}$.
  In the closing up case, it follows from the construction of $M$ that there is an
  $\eta \in \Gamma$ sending $R_1$ to $\alpha R_2$ such that $\Gamma = \la  \Gamma_1, \eta \ra$;
  we set $\phi = \alpha^{-1} \circ \eta$.
  In both cases, we have constructed an isometry $\phi\: \Hy_{f_1} \to \Hy_{f_2}$ sending 
  $R_1$ to $R_2$ which by construction induces the isometry $\varphi\: N_1 \to N_2$.
  In other words, $\phi$ is an extension of $\varphi$.
  
  Let $\Phi\: \aPO_{f_1,\R} \to \aPO_{f_2,\R}$ denote the isomorphism induced via conjugation by $\phi$.
  Let $k_\varphi$ be the field of definition of $\varphi$ (and recall that $K$ denotes the trace field of 
  $\Gamma$).
  By construction, $\Phi^{-1}$ sends the (almost) Zariski-dense subgroup 
  $\Gamma_2 \subset \aPO_{f_2}(K)$ into $\Gamma \subset \aPO_{f_1}(K)$.
  This shows that $\Phi^{-1}$ (hence also $\Phi$) is defined over $K$, and we get 
  $k_\varphi \subset K$.
  
  We now show the reverse inclusion: $K \subset k_\varphi$.
  For the interbreeding case, it suffices to contemplate the generating set 
  for $\Gamma$ given above to note that 
  $\Gamma \subset \aPO_{f_2}(k_\varphi)$, which implies $K \subset k_\varphi$.
  
  Finally for the closing up case we need to show that $\eta \in \aPO_{f_1}(k_\varphi)$.
  Since $\aPO_{f_1}$ is adjoint, it is equivalent to show that the isomorphism
  $\bar \eta\: \aPO_{f_1,\R} \to \aPO_{f_1,\R}$ induced via 
  conjugation is defined over $k_\varphi$.
  The $\alpha$ from above induces via conjugation a map $\bar \alpha\: \aPO_{f_2,\R} \to \aPO_{f_1,\R}$ 
  which sends $\Gamma_2$ to $\Gamma_1$.
  It is therefore defined over $k_1 = k_2$, and thus also over $k_\varphi$.
  Since $\eta = \alpha^{-1} \circ \phi$, we see that $\bar \eta$ is also defined over $k_\varphi$.
  This completes the proof of the closing up case, and thus also that of 
  Theorem~\ref{th:tf-gluing-field-of-def-gluing-isom}. 
\end{proof}

\section{The classical gluings} \label{sec:classical-gluings}

In this section we compute the trace field of the \emph{classical} gluings of Gromov and Piatetski-Shapiro 
\cite{GPS} and of their generalizations \cite{Raimbault, GL}.
In these constructions, the gluing isometry is usually not specified.
For nonarithmeticity, this is not needed, since gluing pieces of non-commensurable arithmetic manifolds 
always results in nonarithmetic manifolds (see the commensurability criterions \cite[§1.6]{GPS}, \cite[Prop.~2.1]{Raimbault} and \cite[Prop.~3.3]{GL}); we will give a proof of this using the trace field in Proposition~\ref{prop:gluing-containing-noncommensurable-nonarith}.
For trace field computations however, the knowledge of the gluing isometry is crucial.
Hence we first define \emph{canonical} gluing isometries for arithmetic pieces of a given shape.

Let $k \subset \R$ be a totally real number field, and let $f_0 = f_0(x_0, \dots, x_{n-1})$ 
be a quadratic form over $k$.
Assume $f_0$ has signature $(n,1)$ and that ${}^\sigma f_0$ is positive definite for all non-trivial 
embeddings $\sigma\:k \inj \R$ (we will say that $f_0$ is \emph{admissible}).
For any totally positive $a \in k$ set $f_a = f_0 + a x_n^2$.
From the conditions on $f_0$ and $a$, it follows that the image $\Gamma_a$ of the group
$\aO_{f_a}(\calO_k) = \aO_{f_a}(k) \cap \aGL_{n+1}(\calO_k)$ in $\aPO_{f_a}$ defines an arithmetic lattice 
in $\aPO_{f_a}(\R)$.
The hyperplane $R_a = \{x_n = 0\} \subset \Hy_{f_a}$ projects onto an 
immersed hypersurface $N_a = \Stab_{\Gamma_a}(R_a) \bs R_a$ in the orbifold $\Gamma_a \bs \Hy_{f_a}$.
More precisely, if $\Gamma_0$ is the image of $\aO_{f_0}(\calO_k)$ in $\aPO_{f_0}$, then:
\begin{enumerate}
  \item The map $\Hy_{f_0} \to R_a; x \mapsto (x,0)$ induces an orbifold isometry 
    \[
      \Gamma_0 \bs \Hy_{f_0} \cong \Stab_{\Gamma_a}(R_a) \bs R_a.
    \]
\end{enumerate}
Let $S_a$ be another hyperplane lift of $N_a$.
Then there exists $\gamma \in \Gamma_a$ with $\gamma S_a = R_a$, and thus:
\begin{enumerate}[resume]
  \item The map $S_a \to R_a; x \mapsto \gamma x$ induces an isometry 
    \[
      \Stab_{\Gamma_a}(S_a) \bs S_a \to \Stab_{\Gamma_a}(R_a) \bs R_a \cong \Gamma_0 \bs \Hy_{f_0}.
    \]
\end{enumerate}

Now for $a, b \in k$, assume that $\Gamma_a' \subset \Gamma_a$ and $\Gamma_b' \subset \Gamma_b$ 
are normal finite-index subgroups that are 
\emph{compatible}, in the sense that there is a fixed finite index subgroup 
$\Gamma_0' \subset \Gamma_0$ such that (1) holds for $\Gamma_0', \Gamma_a'$ and $\Gamma_0',\Gamma_b'$ in place 
of $\Gamma_0, \Gamma_a$.
This holds for instance for principal congruence subgroups of common level.
By normality, it is easy to see that (2) holds as well, with $\Gamma_a'$ (resp.\ $\Gamma_b'$) 
in place of $\Gamma_a$, for any choice of $\gamma \in \Gamma_a$ (resp.\ $\Gamma_b$) with 
$\gamma S_a = R_a$ (resp.\ $\gamma S_b = R_b$).
This gives (at least) one isometry $N_a' \cong \Gamma_0' \bs \Hy_{f_0} \cong N_b'$ for each hypersurfaces
$N_a' \subset \Gamma_a' \bs \Hy_{f_a}$ and $N_b' \subset \Gamma_b' \bs \Hy_{f_b}$ lifting $N_a$ and $N_b$ 
respectively.
We call such an isometry \emph{canonical}.

Next we define fundamental arithmetic pieces.
Let $\Lambda_a$ be a finite-index principal congruence subgroup of $\Gamma_a$; we can choose it such that:
\begin{enumerate}
  \item $\Lambda_a$ is torsion-free (Selberg's lemma) and $\Lambda_a \bs \Hy_{f_a}$ is orientable.
  \item The hyperplane $R_a$ projects down to an embedded orientable hypersurface $N_a'$ in 
    $\Lambda_a\bs \Hy_{f_a}$, by \cite[Lem.~3.1]{BT}.
\end{enumerate}
We will call \emph{fundamental arithmetic piece} a manifold obtained as the completion of 
$\Gamma_a' \bs \Hy_{f_a} - N_1 \cup \cdots \cup N_r$ where $\Gamma_a'$ is a subgroup of 
$\Lambda_a$, normal in $\Gamma_a$, and $N_i$ are disjoint lifts of $N_a'$ all isometric to $N_a'$.
By \cite[Prop.~4.3]{GL}, we can find fundamental arithmetic pieces with any even number of boundary 
components.
For later reference, we will also call the manifold $\Gamma_a' \bs \Hy_{f_a}$ a \emph{fundamental arithmetic manifold}.

Finally, for $i \in \{1, \dots, r\}$ let $a_i \in k$ be totally positive and let $M_i$ be a fundamental 
arithmetic piece of $\Gamma_{a_i}' \bs \Hy_{f_{a_i}}$.
By choosing a common level for the congruence subgroups $\Lambda_{a_i}$ above, one can ensure that 
all $\Gamma_{a_i}'$ are pairwise compatible.
This means that there is a fixed $\Gamma_0' \subset \Gamma_0$ (which is a congruence subgroup) such that 
the map $\Gamma_0' \bs \Hy_{f_0} \to \Gamma_{a_i} \bs \Hy_{f_{a_i}}$ is an isometric embedding.
In this context there are canonical isometries between all the boundary hypersurfaces of the $M_i$.
\begin{theorem} \label{th:tf-cannonical-arithmetic-gluings}
  Let $M$ be a gluing involving fundamental arithmetic pieces of the $\Gamma_{a_i}' \bs \Hy_{f_{a_i}}$.
  Assume all the gluing isometries are canonical. 
  Then the trace field of $M$ is 
  \[
    K = k(\sqrt{a_i a_j} \sep 1 \leq i, j \leq r) = k(\sqrt{a_1 a_2}, \dots, \sqrt{a_1 a_r}).
  \]  
\end{theorem}
\begin{proof}
  At each gluing step, an extension of the gluing isometry is:
  \[
    \begin{tikzcd}
      \phi\:\Hy_{f_{a_i}} \rar["\gamma_i"] & \Hy_{f_{a_i}} \rar["\psi"] & \Hy_{f_{a_j}} \rar["\gamma_j"] & 
      \Hy_{f_{a_j}}
    \end{tikzcd}
  \]  
  where $\gamma_i \in \Gamma_{a_i}, \gamma_j \in \Gamma_{a_j}$ and $\psi$ is induced by 
  $x \mapsto Ax$, where $A$ is the matrix 
  \[
    A = \begin{pmatrix}
      I & 0 \\
      0 & \sqrt{a_j / a_i}
    \end{pmatrix}.
  \]  
  If the gluing happens between the manifolds $M_1$ and $M_2$ (of trace fields $k_1$ and $k_2$ respectively), then 
  the field of definition of $\phi$ is clearly the composition of the fields $k_1$, $k_2$ and $k(\sqrt{a_j/a_i}) = k(\sqrt{a_ia_j})$.
  The theorem follows by induction.
\end{proof}

Two particular instances of this theorem are:
\begin{corollary}\label{cor:GPS}
  Let $M$ be a gluing of Gromov--Piatetski-Shapiro \cite{GPS}. 
  Assume the pieces used are fundamental arithmetic pieces constructed using the quadratic 
  forms $f_a$ and $f_b$ defined over $k$, and that the gluing isometry is canonical.
  Then the trace field of $M$ is $k(\sqrt{ab})$.
\end{corollary}
\begin{corollary}\label{cor:GL}
  Let $M$ be a gluing of Gelander--Levit \cite{GL}.
  Assume the building blocks are fundamental arithmetic pieces constructed using the quadratic 
  forms $f_{a_1}, \dots, f_{a_6}$ defined over $k$, and that the gluing isometries are canonical.
  Then the trace field of $M$ is $k(\sqrt{a_1a_2}, \dots, \sqrt{a_1a_6})$.
\end{corollary}

The key fact is that Theorem~\ref{th:tf-cannonical-arithmetic-gluings} can be used with pieces of commensurable manifolds to 
produce nonarithmetic manifolds.

\begin{theorem} \label{th:realization-trace-fields}
  Let $K = \Q(\sqrt{a_1}, \dots, \sqrt{a_r})$ be an arbitrary totally real multiquadratic extension of $\Q$.
  Then for $n \equiv 2 \pmod 4$ there exists an $n$-manifold $M$ of trace field $K$ obtained as a 
  gluing of pieces of \emph{pairwise commensurable} arithmetic manifolds of trace field $\Q$.
\end{theorem}
\begin{proof}
  Consider the quadratic forms 
  \[
  f_0 = - x_0^2 + x_1^2 + \cdots + x_{n-1}^2 \quad \text{and} \quad f_i = f_0 + a_i x_n^2.
  \]
  Theorem~\ref{th:tf-cannonical-arithmetic-gluings} implies the existence of a 
  gluing $M$ with trace field $K$ constructed from fundamental arithmetic pieces using 
  the quadratic forms $f_i$ with canonical gluing isometries.
  Thus we only need to show that all the arithmetic groups are commensurable, i.e., 
  that the forms $f_i$ are similar.

  It suffices to show that $a f_0 \cong f_0$ for any positive $a \in \Q$, since then we have 
  $a_j f_i \cong a_i f_0 + a_i a_j x_n^2 = a_i f_j$.
  The discriminants of $f_0$ and $af_0$ coincide, since the number of variables is even.
  Moreover, setting $\delta_n = -1$ and $\delta_i = 1$ for $i < n$, one has for any prime $p$:
  \[
  \epsilon_p(a f_0) = \prod_{i<j} (a \delta_i, a \delta_j)_p = 
  (a,a)_p ^{\frac{(n-1)(n-2)}{2}} \, (a,-a)_p^{n-1}.
  \]
  The term $(a, -a)_p$ equals 1 by standard properties of the Hilbert symbol, and since 
  $n \equiv 2 \pmod 4$, the exponent of $(a, a)_p$ is even.
  Therefore $\epsilon_p(a f_0) = 1 = \epsilon_p(f_0)$ for any prime $p$ and any $a$.
\end{proof}

We end this section by giving both compact and non-compact examples in any even dimension of gluings of pieces of commensurable 
arithmetic manifolds realizing trace fields of arbitrary large degree.
These together with Theorem~\ref{th:realization-trace-fields} imply Theorem~\ref{th:intro-commensurable-large-trace-field} of 
the introduction.

\begin{example}\label{ex:gluings-Q-0-mod-4}
  Let $n \equiv 0 \pmod 4$ and let $f_0$ be the quadratic form from the proof of 
  Theorem~\ref{th:realization-trace-fields}.
  For $a \in \Z$ we have $\epsilon_p(af_0) = (a,a)_p (a,-a)_p = (a, -1)_p = \left(\frac{-1}{p}\right)^{v_p(a)}$ 
  for each prime $p$.
  For $r \geq 1$, let $a_1, \dots, a_r$ be distinct primes $\equiv 1 \pmod 4$; they exist by Dirichlet's 
  Theorem on arithmetic progressions.
  Such $a_i$ fulfill $\left(\frac{-1}{a_i}\right) = 1$, and it follows that $a_i f_0 \cong f_0$.
  By the same arguments as in the proof of Theorem~\ref{th:realization-trace-fields} one obtains a gluing $M$ of trace field 
  $\Q(\sqrt{a_1}, \dots, \sqrt{a_r})$ of degree $2^r$, for any $r \geq 1$.
\end{example}

\begin{example}\label{ex:gluings-Qsqrt2}
  Let $k = \Q(\sqrt{2})$ and consider for even $n = 2m$ the quadratic form 
  $f_0 = -\sqrt{2} x_0^2 + x_1^2 + \cdots + x_{n-1}^2$.
  Let $A = \{a_1, a_2, \dots\}$ be the set of primes in $k$ that split in $k(\sqrt{-\sqrt{2}})$ (resp.\ 
  $k(\sqrt{\sqrt{2}})$) if $n \equiv 0 \pmod 4$ \linebreak (resp.\ $n \equiv 2 \pmod 4$).
  Then $A$ is infinite (a consequence of Chebotarev's Density Theorem).
  We claim that $a_i f_0 \cong f_0$ for every $i \geq 1$.
  Indeed, as in the proof of Theorem~\ref{th:realization-trace-fields}, the discriminants of $a_i f_0$ and 
  $f_0$ coincide, and for each place $v$ of $k$, we have 
  \[
    \epsilon_v(a_i f_0) = (-\sqrt{2} a_i, a_i)_v^{n-1} \, (a_i, a_i)_v^{\frac{(n-1)(n-2)}{2}}.
  \]
  If $n \equiv 0 \pmod 4$, this quantity equals $(- \sqrt{2}a_i, a_i)_v (a_i, a_i)_v = (-\sqrt{2}, a_i)_v$.
  Since $a_i$ splits in $k(\sqrt{-\sqrt{2}})$, we have $a_i = b^2 + \sqrt{2} c^2$ for some $b,c \in k$. \linebreak
  Hence the equation $-\sqrt{2} x^2 + a_iy^2 = z^2$ has a solution in $k$, and thus $(-\sqrt{2}, a_i)_v = 1$ for each place $v$.
  Similarly, if $n \equiv 2 \pmod 4$, we have $\epsilon_v(a_i f_0) = (\sqrt{2}, a_i)_v = 1$ for each place $v$.
  Proceeding again as in the proof of Theorem~\ref{th:realization-trace-fields}, we 
  get a manifold $M$ of trace field $k(\sqrt{a_1}, \dots, \sqrt{a_r})$ of degree $2^r$, for any 
  $r \geq 1$.
\end{example}

\section{Corollaries and auxiliary results} \label{sec:corollaries}

This section is devoted to general results and consequences of Theorem~\ref{th:tf-gluing-field-of-def-gluing-isom} about gluing 
constructions involving arbitrary manifolds.
We start with a proposition which follows actually directly from Lemma~\ref{lem:trace-field-inclusion}.
\begin{proposition} \label{prop:gluing-containing-noncommensurable-nonarith}
  Let $M$ be a gluing involving (among others) two manifolds with boundary $M_1$ and $M_2$. 
  Assume that $M_1$ and $M_2$ are pieces of non-commensurable arithmetic manifolds.
  Then $M$ is nonarithmetic.
\end{proposition}
\begin{proof}
  For $i = 1,2$, let $k_i$ denote the trace field of $M_i$ and let $\aPO_{f_i}$ denote its ambient group.
  Then $k_i$ (resp.\ $\aPO_{f_i}$) is the trace field (resp.\ the ambient group) of the arithmetic manifold 
  $M_i$ is a piece of, see \cite[Lem.~2.5]{Emery-Mila}.
  By Lemma~\ref{lem:trace-field-inclusion}, if $k$ denotes the trace field of $M$ then $k_1$ and $k_2$ are subfields of $k$ and 
  $\aPO_{f_1,k} \cong \aPO_{f_2,k}$.
  Now if $M$ is arithmetic, its ambient group is admissible.
  This forces $k = k_1 = k_2$ and $\aPO_{f_1} \cong \aPO_{f_2}$, and thus the corresponding arithmetic 
  manifolds are commensurable.
\end{proof}

 The next result controls how large the trace field gets under interbreedings and closing ups.
\begin{theorem} \label{th:gluing-quadratic}
  \begin{enumerate}
    \item Let $M_1$ and $M_2$ be manifolds with boundary, of trace fields $k_1$ and $k_2$ respectively. 
      Let $M$ be obtained by interbreeding $M_1$ with $M_2$.
      Then the trace field of $M$ is contained in a quadratic extension of 
      $k_1 k_2$, the composition of $k_1$ and $k_2$.  
    \item Let $M_1$ be a manifold with boundary of trace field $k_1$.
      Let $M$ be obtained by closing up $M_1$.
      Then the trace field of $M$ is contained in a quadratic extension of $k_1$.
  \end{enumerate}
\end{theorem}
\begin{proof}  
We treat both cases at once by setting $M_2 = M_1$ in the closing up case.
Let $k = k_1k_2$ (resp.\ $k_1$) in the interbreeding (resp.\ closing up) case.
  Let $\tM_i \subset \Hy_{f_i}$ be a model for $M_i$, and let $N_i \subset M_i$ denote the gluing hypersurface.
  Up to choosing a different orthogonal basis for $f_i$, we may assume that $f_i = q_i + a_i x_n^2$ for some quadratic form $q_i$ and that the hyperplane 
  $R_i \subset \Hy_{f_i}$ corresponding to $\{x_n = 0\}$ is a hyperplane lift of $N_i \subset M_i$.
  The isometry $\varphi\: N_1 \to N_2$ is induced by a matrix $A_0 \in \aGL_n(k)$ such that
  $q_2 \circ A_0 = \lambda q_1$, for some $\lambda \in k$.
  Thus up to replacing $f_2$ with $\lambda^{-1}(f_2 \circ A) = q_1 + \lambda^{-1} a_2 x_n^2$, where 
  \[
    A = 
    \begin{pmatrix}
      A_0 & 0 \\
      0 & 1
    \end{pmatrix}
  \]
  and conjugating the monodromy representations accordingly, we can assume that $f_2 = q_1 + a_2 x_n^2$.
  Now by construction, using these models, the matrix 
  \[
    B = \begin{pmatrix}
      I & 0 \\
      0 & \sqrt{a_1 / a_2}.
    \end{pmatrix}
  \]  
  is such that $f_2 \circ B = f_1$, and thus induces an isometry $\phi\:\Hy_{f_1} \to \Hy_{f_2}$ which is an extension of the original gluing isometry $\varphi$.
  Therefore the trace field of $M$ is contained in $k(\sqrt{a_1 / a_2}) = k(\sqrt{a_1 a_2})$, as desired.
\end{proof}

Inductively, we find:
\begin{corollary}
  Let $M$ be a gluing of manifolds with boundary, all of which have trace field $k$.
  Then the trace field of $M$ is a multiquadratic extension of $k$ of degree at most $2^r$, 
  where $r$ is the number of gluing operations (interbreeding or closing up).
\end{corollary}

\begin{remark}
  This gives another proof of Theorem~2.7 in \cite{Emery-Mila}. 
  If the pieces are fundamental pieces with canonical gluing isometries, 
  Theorem~\ref{th:tf-cannonical-arithmetic-gluings} even permits an exact computation of the trace field.
\end{remark}

The next result shows that doubling does not change the trace field.
\begin{proposition}\label{prop:double-same-trace-field}
  Let $M_1$ be a manifold with boundary, and let $M$ be its double.
  Then the trace field of $M$ equals that of $M_1$.
\end{proposition}
\begin{proof}
  The double of $M_1$ can be realized by interbreeding first $M_1$ with a copy of itself, and then 
  closing up the remaining boundary components.
  At each step, we choose the same model in $\Hy_f$ for the two pieces in the interbreeding case 
  (resp.\ one piece in the closing up case) with the same hyperplane lift for the two 
  hypersurfaces that are to be glued together.
  It is then clear that the identity $\Hy_f \to \Hy_f$ is a valid extension of the gluing isometry.
  Thus the trace field remains untouched.
\end{proof}
Since the manifolds constructed in \cite{Agol,BT} are doubles, we get:
\begin{corollary}
  Let $M$ be a manifold of Agol-Belolipetsky-Thomson \cite{Agol,BT}.
  Then its trace field equals the trace field of the arithmetic manifold used in its construction 
  (i.e., it is quasi-arithmetic in the sense of \cite{Vinberg-non-arith-1}).
\end{corollary}
\begin{remark}
  This was already proved in \cite{Thomson}.
  In order to investigate the commensurability of those manifolds, one can use a finer invariant 
  called the \emph{adjoint trace ring}, see \cite{Mila18}.
\end{remark}

The final result of this section shows that closing up does not change the trace field when the dimension is \emph{odd}.
\begin{proposition} \label{prop:odd-dimensional-closing-up}
  Let $M_1$ be an \emph{odd dimensional} manifold with boundary.
  Let $M$ be obtained by closing up $M_1$ at any two isometric boundary components.
  Then the trace field of $M$ equals that of $M_1$.
\end{proposition}
\begin{proof}
  Let $\tM \subset \Hy_{f}$ be a model for $M_1$ with corresponding monodromy representation 
  $\Gamma \subset \aPO_f(k)$, where $k$ is the trace field of $M_1$.
  Let $N_1, N_2$ be two isometric boundary components of $M_1$ with hyperplane lifts $R_1, R_2 \subset 
  \bd \tM$
  and let $\varphi\: N_1 \to N_2$ be a gluing isometry.
  Let $\phi \in \aPO_f(\R)$ be an extension of $\varphi$ such that $\phi(R_1) = R_2$.
  By Theorem~\ref{th:tf-gluing-field-of-def-gluing-isom}, it is enough to show that $\phi \in \aPO_f(k)$. 
  
  Let $\Phi\: \aPO_{f,\R} \to \aPO_{f,\R}$ denote the isomorphism induced via conjugation by $\phi$.
  Let $\sigma$ be a $k$-automorphism of $\C$. 
  Then ${}^\sigma \Phi$ is induced via conjugation by $\sigma(\phi)$.
  Let $\G_0$ denote the Zariski-closure of $\Stab_\Gamma(R_1)$ in $\aPO_f$.
  As $\Stab_\Gamma(R_1) \subset \aPO_f(k)$, the maps $\Phi$ and ${}^\sigma \Phi$ must agree on it.
  Therefore, $\Phi^{-1} \circ {}^\sigma \Phi$ is the identity on $\Stab_\Gamma(R_1)$, and thus also on 
  $\G_0$ by Zariski-density.
  It follows that $\phi^{-1} \sigma(\phi)$ commutes with every element of $\G_0$.
  Since $\aPO_f$ is adjoint, this implies $\sigma(\phi) = \phi$ or $\sigma(\phi) = \phi\rho$ 
  (where $\rho$ is the reflection at $R_1$).
  Moreover, since $\rho \in \aPO_f(k)$ by Lemma~2.1 of \cite{Emery-Mila}, we have $\sigma(\phi \rho) = \sigma(\phi) \rho$ and thus 
  $\sigma$ acts on the set $\{\phi, \phi\rho\}$.

  Since $n$ is odd, the algebraic group $\aPO_f$ has two connected components, and 
  $\aPO_f^\circ(\R)$ consists of the orientation-preserving isometries of $\Hy_f$.
  It follows that the set $\aPO_f^\circ \cap \{\phi, \phi \rho\}$ contains exactly one element.
  Now as $\aPO_f^\circ$ is defined over $k$, it is stable under $\sigma$.
  Thus $\sigma$ fixes one, and hence both elements of $\{\phi, \phi\rho\}$,
  i.e., $\sigma(\phi) = \phi$.
  Since $\sigma$ was arbitrary, we get $\phi \in \aPO_f(k)$, as desired.
\end{proof}

\section{Examples and applications} \label{sec:examples}

Let $M$ be a manifold containing a non-separating hypersurface $N$.
Recall that a \emph{twist} of $M$ (at $N$) is any manifold obtained by closing up the piece $M_1$ of $M$ defined as the completion of
$M-N$.

The goal of this section is twofold.
First we produce examples where twisting arithmetic manifolds increases the trace field, as promised in 
Theorem~\ref{th:intro-twist-nonarithmetic}.
In particular, these twists will all be nonarithmetic.
Second, we will use the results from the preceding sections to show that the Coxeter simplex $\Delta_5$ defined in the 
introduction is not commensurable to any arithmetic gluing, thereby proving Theorem~\ref{th:intro-Delta-5}.

The tool we use to construct twists is the following:

\begin{lemma} \label{lem:closing-up-increase-trace-field}
  Let $f_0(x_0, \dots, x_{n-1})$ be an admissible quadratic form over a totally real number field $k \subset \R$, 
  and let $a \in k^\times - (k^\times)^2$ be totally positive such that $f_0 \cong a f_0$.
  Assume there exists a matrix $A_0 \in \aGL_n(k)$ such that 
  \begin{enumerate}
    \item $f_0 \circ A_0 = a f_0$,
    \item $\frac{1}{a} A_0^2 \in \aO_{f_0}(\calO_k)$.
  \end{enumerate}
  Then there exists a twist $M$ of the arithmetic manifold corresponding to $f = f_0 + x_n^2$ having trace field $k(\sqrt{a})$.
  In particular, $M$ is nonarithmetic.
\end{lemma}
\begin{proof}
  Let $\Gamma \bs \Hy_{f}$ be a fundamental arithmetic manifold constructed using $f$ (see 
  Section~\ref{sec:classical-gluings}).
  Let $R = \{x_n = 0 \} \subset \Hy_f$; by definition, there is an arithmetic manifold 
  $\Lambda \bs \Hy_{f_0}$ such that the map $\Lambda \bs \Hy_{f_0} \to \Gamma \bs \Hy_f$ 
  is an isometric embedding, and we can assume that its image is non-separating.
  
  Let $g$ denote the image of $\frac{1}{\sqrt{a}}A_0$ in $\aPO_{f_0}$.
  Since $g^2$ is in the image of $\aO_{f_0}(\calO_k)$, we have $g^m \in \Lambda$ for a suitable power $m$, 
  and thus the group $\Lambda' = \Lambda \cap g \Lambda g^{-1} \cap \cdots \cap g^{m-1} \Lambda g^{-m+1}$
  is a finite index subgroup of $\Lambda$ that is normalized by $g$.
  Hence $g$ induces an isometry $\varphi$ of $N' = \Lambda' \bs \Hy_{f_0}$.
  Since $\Lambda'$ is geometrically finite, it is separable in $\Gamma$ by \cite[Cor.~1.12]{BHW} and 
  thus there exists a finite index subgroup $\Gamma' \subset \Gamma$ such that the map 
  $\Lambda' \bs \Hy_{f_0} \to \Gamma' \bs \Hy_f$ is an isometric embedding.
  Call $N$ its image.

  Let $M_1$ be the piece obtained as the completion of $\Gamma' \bs \Hy_{f} - N$.
  Let $p\:\Hy_f \surj \Gamma' \bs \Hy_f$ be the covering map and let $\tM_1, \tM_1' \subset \Hy_f$ 
  be the two connected components of $p^{-1}(M_1)$ containing $R$ in their boundary.
  Then $\tM_1$ and $\tM_1'$ are two models for $M_1$ and if $N_1, N_1'$ are the boundary components of $M_1$, 
  we have that $R$ is a hyperplane lift of $N_1$ in $\tM_1$ and of $N_1'$ in $\tM_1'$ (or vice versa).

  Let $M$ be obtained via closing up $M_1$ using the isometry $\varphi$.
  It is clear that the isometry $\phi \in \aPO_f(k(\sqrt{a}))$ defined by the matrix
  \[
    \begin{pmatrix}
      \frac{1}{\sqrt{a}} A_0 & 0 \\
      0 & 1
    \end{pmatrix}
  \]
  is an extension of $\varphi$.
  If $\sigma$ is the non-trivial $k$-automorphism of $k(\sqrt{a})$, we have $\sigma(\phi) = \phi \rho$, 
  where $\rho$ is the reflection at the hyperplane $R$, and thus $\phi \notin \aPO_f(k)$.
  Hence the gluing isometry $\varphi$ has field of definition $k(\sqrt{a})$, and by Theorem~\ref{th:tf-gluing-field-of-def-gluing-isom}, this is also the trace field of $M$.
\end{proof}

In even dimensions, one can use the lemma to easily construct twists of arithmetic manifolds which increase the trace field by a quadratic extension.
We will use the notation $\< a_1, \dots, a_m\>$ for the quadratic form $a_1 x_1^2 + \cdots + a_m x_m^2$ and $\perp$ for orthogonal sum.

The first example shows how to realize $\Q(\sqrt{d})$, $d$ odd, as the twist of an arithmetic manifold with trace field $\Q$, 
in any even dimension.
\begin{example}\label{ex:1}
  Let $d >0$ be odd and square free. Write $d = 2 b +1$ and consider the following quadratic forms and matrices:
  \[
  q_1 = \<-1, 1\>, \; 
  A_1 = \begin{pmatrix}
    b+1 & b \\
    -b & -(b+1)
  \end{pmatrix} \text{ and }
  q_2 = \< d, 1 \>,\;  A_2 = \begin{pmatrix}
    0 & 1 \\
    d & 0
  \end{pmatrix}.
\]  
It is easily computed that $q_i \circ A_i = d \,q_i$ and $A_i^2 = d I$.
Therefore, one can apply Lemma~\ref{lem:closing-up-increase-trace-field} to the quadratic form 
$f_0 = q_1 \perp q_2 \perp \cdots \perp q_2$, and the matrix $A_0$ obtained as a block-diagonal matrix with diagonal 
entries $A_1, A_2, \dots, A_2$. 
The resulting twist $M$ then has trace field $\Q(\sqrt{d})$.
\end{example}

Table~\ref{tab:1} shows how to adapt the previous example to produce $\Q(\sqrt{d})$ for even $d \leq 42$ (in even dimensions $\geq 4$).
This suggests that it is in fact possible to realize every quadratic extension of $\Q$ as twists of arithmetic manifolds of trace 
field $\Q$.
In the second example, we show how to realize $k(\sqrt{b})$, $b \in k \setminus \Q$ totally positive, as the twist of an arithmetic manifold
with trace field $k$.
\begin{example}\label{ex:2}
  Let $k$ be a totally real number field different from $\Q$, and let $b \in k\setminus \Q$ be totally positive.
  Up to scaling $b$ with a square, we can assume that $b-1$ has exactly one negative conjugate.
  Consider the following quadratic forms and matrices:
  \[
  q_1 = \< b-1 ,1 \>, \; A_1 = \begin{pmatrix}
    1 & 1 \\
    b-1  & -1
  \end{pmatrix} \text{ and } 
  q_2 = \<b , 1 \>, \; A_2 = \begin{pmatrix}
    0 & 1 \\
    b & 0
  \end{pmatrix}.
\]
  As before, we have $q_i \circ A_i = b q_i$ and $A_i^2 = b I$, and we can hence proceed as in the previous example.
  Observe that the conditions on $b$ ensure that $f_0$ is admissible.
  The resulting twist $M$ has trace field $k(\sqrt{b})$.
\end{example}

These two examples imply Theorem~\ref{th:intro-twist-nonarithmetic} of the introduction.
The end of this section is devoted to the proof that the orbifold $\Delta_5 \bs \Hy^5$ is not commensurable to any gluing of arithmetic
pieces (Theorem~\ref{th:intro-Delta-5}).
The main tool is the following:

\begin{theorem}\label{th:arith-gluing-odd-dim}
  Let $M \in \Agl$ be an \emph{odd dimensional} gluing of arithmetic pieces.
  Assume all the pieces used in its construction come from commensurable arithmetic manifolds, of trace field $k$.
  Then the trace field of $M$ equals $k$ (it is \emph{quasi-arithmetic} in the sense of \cite{Vinberg-non-arith-1}).
\end{theorem}
\begin{proof}
  It is enough to show that using such pieces, the trace field does not change under interbreedings and closing ups.
  Since the dimension is odd, closing up does not change the trace field (by Proposition~\ref{prop:odd-dimensional-closing-up}). 
  Assume that $M$ is constructed from interbreeding $M_1$ with the arithmetic piece $M_2$.
  For $i = 1,2$, let $\tM_i \subset \Hy_{f_i}$ be a model for $M_i$.
  From the proof of Theorem~\ref{th:gluing-quadratic}, we can assume that $f_i = q + a_i x_n^2$ for a fixed quadratic form $q$.
  Now since $M_2$ is commensurable to an arithmetic piece in $M_1$, the forms $f_1$ and $f_2$ must be similar.
  Since $n+1$ is even, $f_1$ and $f_2$ must have the same discriminants, and this implies that $a_1/a_2$ is a square in $k$.
  Therefore the extension used in the proof of Theorem~\ref{th:gluing-quadratic} is defined over $k$ and the trace field of $M$ is $k$, by Theorem~\ref{th:tf-gluing-field-of-def-gluing-isom}.
\end{proof}

We are ready for the:
\begin{proof}[Proof of Theorem~\ref{th:intro-Delta-5}]
  Using their Angle Rigidity Theorem \cite[Th.~4.1]{FLMS}, Fisher et.~al.\ prove that the orbifold $\Delta_5 \bs \Hy^5$ is not 
  commensurable 
  to any gluing of pieces of non-commensurable arithmetic manifolds (see \cite[§6.2]{FLMS}).
  Assume it is commensurable to the gluing $M$ of pieces of commensurable arithmetic manifolds $M_1, \dots, M_r$.
  As shown for example in \cite{Emery-Mila}, the orbifold $\Delta_5 \bs \Hy^5$ (and thus $M$) has trace field $\Q(\sqrt{2})$ and 
  ambient group $\aPO_f$ with $f = -x_0^2 + x_1^2 + \cdots + x_5^2$.
  By Lemma~\ref{lem:trace-field-inclusion}, the trace field of the pieces (and thus of the $M_i$ by Lemma~2.5 of 
  \cite{Emery-Mila}) is $\Q$ or $\Q(\sqrt{2})$.
  Since $\aPO_{f,\Q(\sqrt{2})}$ is not admissible, it must be $\Q$.
  By Theorem~\ref{th:arith-gluing-odd-dim}, $M$ (and thus $\Delta_5 \bs \Hy^5$) must have trace field $\Q$ as well, a contradiction.
\end{proof}


\begin{table}[t]\label{tab:1}
  \centering
  \caption{Twists realizing $\Q(\sqrt{d})$ for even $d \leq 42$.}
\newcolumntype{L}{>{$}l<{$}} 
\renewcommand{\arraystretch}{1.2} 
\begin{tabular}{LLLL}
  \toprule
  $Trace Field$ & $Quadratic forms$ & $Matrices$ \\
  \midrule
K=\Q(\sqrt{2}) & 
\begin{array}{l}
q_1 = \<-1, 1, 1, 1\> \\ q_2 = \<1, 1\> \end{array} & 
A_1 = \begin{psmallmatrix}
2 & 0 & 1 & 1 \\
0 & 0 & 1 & -1 \\
1 & 1 & 1 & 1 \\
1 & -1 & 1 & 1
\end{psmallmatrix} & A_2 = \begin{psmallmatrix}
1 & 1 \\
1 & -1
\end{psmallmatrix} \\

K=\Q(\sqrt{6}) & 
\begin{array}{l}
q_1 = \<-1, 1, 1, 2\> \\ q_2 = \<1, 2\> \end{array} & 
A_1 = \begin{psmallmatrix}
3 & 0 & 1 & 2 \\
0 & 0 & -2 & 2 \\
1 & 2 & 1 & 2 \\
1 & -1 & 1 & 2
\end{psmallmatrix} & A_2 = \begin{psmallmatrix}
2 & 2 \\
1 & -2
\end{psmallmatrix} \\

K=\Q(\sqrt{10}) & 
\begin{array}{l}
q_1 = \<-1, 1, 1, 1\> \\ q_2 = \<1, 1\> \end{array} & 
A_1 = \begin{psmallmatrix}
4 & 1 & 1 & 2 \\
1 & 3 & -1 & 1 \\
-1 & 1 & 1 & -3 \\
2 & 1 & 3 & 2
\end{psmallmatrix} & A_2 = \begin{psmallmatrix}
1 & 3 \\
3 & -1
\end{psmallmatrix} \\

K=\Q(\sqrt{14}) & 
\begin{array}{l}
q_1 = \<-1, 1, 1, 5\> \\ q_2 = \<1, 5\> \end{array} & 
A_1 = \begin{psmallmatrix}
4 & 1 & 1 & 0 \\
-1 & -1 & -3 & 5 \\
-1 & -3 & -1 & -5 \\
0 & 1 & -1 & -2
\end{psmallmatrix} & A_2 = \begin{psmallmatrix}
3 & 5 \\
1 & -3
\end{psmallmatrix} \\

K=\Q(\sqrt{22}) & 
\begin{array}{l}
q_1 = \<-1, 1, 1, 2\> \\ q_2 = \<1, 2\> \end{array} & 
A_1 = \begin{psmallmatrix}
5 & 0 & 1 & 2 \\
0 & -4 & 2 & -2 \\
-1 & 2 & 1 & -6 \\
-1 & -1 & -3 & -2
\end{psmallmatrix} & A_2 = \begin{psmallmatrix}
2 & 6 \\
3 & -2
\end{psmallmatrix} \\

K=\Q(\sqrt{26}) & 
\begin{array}{l}
q_1 = \<-1, 1, 1, 1\> \\ q_2 = \<1, 1\> \end{array} & 
A_1 = \begin{psmallmatrix}
6 & 0 & 1 & 3 \\
0 & -4 & 3 & -1 \\
-1 & 3 & 3 & -3 \\
-3 & -1 & -3 & -5
\end{psmallmatrix} & A_2 = \begin{psmallmatrix}
1 & 5 \\
5 & -1
\end{psmallmatrix} \\

K=\Q(\sqrt{30}) & 
\begin{array}{l}
q_1 = \<-1, 1, 1, 5\> \\ q_2 = \<1, 5\> \end{array} & 
A_1 = \begin{psmallmatrix}
6 & 0 & 1 & 5 \\
0 & 0 & 5 & -5 \\
-1 & 5 & -1 & -5 \\
-1 & -1 & -1 & -5
\end{psmallmatrix} & A_2 = \begin{psmallmatrix}
5 & 5 \\
1 & -5
\end{psmallmatrix} \\

K=\Q(\sqrt{34}) & 
\begin{array}{l}
q_1 = \<-1, 1, 1, 1\> \\ q_2 = \<1, 1\> \end{array} & 
A_1 = \begin{psmallmatrix}
6 & 0 & 1 & 1 \\
0 & 4 & 3 & -3 \\
-1 & 3 & -5 & -1 \\
-1 & -3 & -1 & -5
\end{psmallmatrix} & A_2 = \begin{psmallmatrix}
3 & 5 \\
5 & -3
\end{psmallmatrix} \\

K=\Q(\sqrt{38}) & 
\begin{array}{l}
q_1 = \<-1, 1, 1, 2\> \\ q_2 = \<1, 2\> \end{array} & 
A_1 = \begin{psmallmatrix}
7 & 0 & 3 & 2 \\
0 & -4 & 2 & -6 \\
-3 & 2 & -5 & -6 \\
-1 & -3 & -3 & 2
\end{psmallmatrix} & A_2 = \begin{psmallmatrix}
6 & 2 \\
1 & -6
\end{psmallmatrix} \\

K=\Q(\sqrt{42}) & 
\begin{array}{l}
q_1 = \<-1, 1, 1, 6\> \\ q_2 = \<1, 6\> \end{array} & 
A_1 = \begin{psmallmatrix}
7 & 0 & 1 & 6 \\
0 & 0 & 6 & -6 \\
-1 & 6 & -1 & -6 \\
-1 & -1 & -1 & -6
\end{psmallmatrix} & A_2 = \begin{psmallmatrix}
6 & 6 \\
1 & -6
\end{psmallmatrix} \\
\bottomrule
\end{tabular}
\end{table}

{\small
\bibliographystyle{abbrv}
\bibliography{bibliography} 
}

\end{document}